\documentclass[12pt]{amsart}
\usepackage[top=1.25in,left=1in,bottom=1.25in,right=1in]{geometry}
\usepackage{stmaryrd}
\usepackage{amssymb,amsmath,amsthm,latexsym,verbatim,bbm}
\usepackage[colorlinks=false,pdfborder={0 0 0}]{hyperref}
\usepackage{enumerate}

\begin{document}

\newtheorem{theorem}{Theorem}[section]
\newtheorem*{theorem*}{Theorem}
\newtheorem{corollary}[theorem]{Corollary}
\newtheorem*{corollary*}{Corollary}
\newtheorem*{remark}{Remark}
\newtheorem*{question}{Question}
\newtheorem*{example}{Example}
\newtheorem*{conjecture}{Conjecture}
\newtheorem{proposition}[theorem]{Proposition}
\newtheorem{lemma}[theorem]{Lemma}
\newtheorem*{lemma*}{Lemma}
\theoremstyle{definition}
\newtheorem{definition}[theorem]{Definition}
\newtheorem*{definition*}{Definition}

\newcommand{\symdiff}{\hspace{0.05em}\mbox{$\triangle$}\hspace{0.05em}}
\newcommand{\normal}{\hspace{0.05em}\triangleleft\hspace{0.05em}}
\newcommand{\actson}{\curvearrowright}

\newcommand{\SLnbb}[1]{\mathrm{SL}_{n}(\mathbb{#1})}
\newcommand{\SL}{\mathrm{SL}}
\newcommand{\SLn}{\mathrm{SL}_{n}}
\newcommand{\PSLnbb}[1]{\mathrm{PSL}_{n}(\mathbb{#1})}
\newcommand{\PSL}{\mathrm{PSL}}
\newcommand{\PSLn}{\mathrm{PSL}_{n}}
\newcommand{\GLnbb}[1]{\mathrm{GL}_{n}(\mathbb{#1})}
\newcommand{\GL}{\mathrm{GL}}
\newcommand{\GLn}{\mathrm{GL}_{n}}

\newcommand{\Haar}{\mathrm{Haar}}
\newcommand{\Ind}{\mathrm{Ind}}
\newcommand{\Symm}{\mathrm{Symm}}
\newcommand{\Comm}{\mathrm{Comm}}
\newcommand{\Har}{\mathrm{Har}}
\newcommand{\proj}{\mathrm{proj}}
\renewcommand{\Re}{\mathrm{Re}}
\newcommand{\bbR}{\mathbb{R}}
\newcommand{\bbZ}{\mathbb{Z}}
\newcommand{\bbone}{\mathbbm{1}}

\newcommand{\supp}[1][]{\ifthenelse{\equal{#1}{}}{\mathrm{supp}~}{\mathrm{supp}~#1}}
\newcommand{\rpf}[2]{#1\hspace{-.175em}\sslash\hspace{-.2em}#2}

\newcommand{\spacedcap}{\hspace{1pt}\cap\hspace{1pt}}

\newcommand{\stab}{\mathrm{stab}}

\newcommand{\arxiv}[1]{\href{http://arxiv.org/abs/#1}{arXiv:#1}}

\title[A Normal Subgroup Theorem for Commensurators]{A Normal Subgroup Theorem for \\ Commensurators of Lattices}
\author{Darren Creutz}\thanks {The authors gratefully acknowledge the support
of the NSF through grant 1007227.}
%{Vanderbilt University \texttt{ darren.creutz@vanderbilt.edu}}
\author{Yehuda Shalom}
%\thanks{University of California: Los Angeles \texttt{ shalom@math.ucla.edu}}
\date{}
\address{Department of Mathematics, Vanderbilt University}
\address{Department of Mathematics UCLA, and School of Mathematical Sciences,
Tel-Aviv University}

\begin{abstract}
We establish a general normal subgroup theorem for commensurators 
of lattices in locally compact groups. While the statement
is completely elementary, its proof, which rests on the 
original strategy of Margulis in the case of higher rank lattices, 
relies heavily on analytic tools pertaining to amenability and
Kazhdan's property (T). It is a counterpart to the
normal subgroup theorem for irreducible lattices of Bader and the second 
named author, and may also be
used to sharpen that result when one of the ambient factors is totally
disconnected.
\end{abstract}

\maketitle

\section{Introduction}

An interesting feature of Margulis' rigidity theory of lattices in semisimple algebraic groups is the parallelism between results in the higher rank setting, 
and those for dense commensurators of lattices with no rank assumption. 
Most notable here is Margulis' superrigidity, 
which was fundamental
to the proof of the arithmeticity theorem of higher rank lattices. 
Along with it was established an analogous rigidity result, key to
the striking arithmeticity of {\it all} lattices with dense 
commensurators (see \cite{Ma91}).

Over the years Margulis' rigidity theory has been widely extended 
to lattices in a general setting, first of more geometric flavor 
(e.g.~\cite{LMZ94}, \cite{burgermozes2}, \cite{burgermozes}) and then, following 
\cite{Sha00}, to the very abstract ``higher rank'' framework
of (irreducible lattices in) products of at least two locally 
compact groups -- 
cf.~\cite{monod}, \cite{GKM08}, \cite{CR09}, \cite{FM09}, \cite{BFGM}, \cite{MS04}, \cite{MS06}. The parallelism between the theories of ``higher rank lattices''  
and ``rank free dense commensurators'' remained (see also ~\cite{remy}), and essentially all rigidity results in the former framework found a counterpart in the latter. 
All but one -- the celebrated 
normal subgroup
theorem (abbreviated NST hereafter). 

In \cite{BS05} a general NST was established by Bader and the second author for irreducible lattices in products of at least two locally compact, compactly generated 
groups (see also 
\cite{burgermozes} in the case of tree lattices). The main purpose
of this paper is to establish an analogous ``rank free'' result for dense commensurators of lattices in one such group. This theorem also allows for a 
sharpening of \cite{BS05} when one of the ambient factors is totally disconnected. 
 As in \cite{BS05}, and following the original spirit of Margulis' NST, 
the precise statement remains
of entirely elementary and purely group theoretic nature, while the proof 
employs heavy analytic techniques pertaining to Kazhdan's property (T) and amenability.

Here is the main result of the paper:
\begin{theorem}[Theorem \ref{T:1} and Proposition \ref{P:oneone}]
Let $G$ be a locally compact, second countable compactly generated group, which is not a compact extension of an abelian group.

\smallskip

Let $\Gamma < G$ be a discrete co-compact subgroup, or more generally, a finitely generated square integrable lattice (see Section  \ref{S:Thalf} below), and $\Lambda < G$ be a dense subgroup 
which contains and commensurates $\Gamma$ (i.e.~$\lambda \Gamma \lambda ^{-1} \cap
\Gamma$ has finite index in both  $\lambda \Gamma \lambda ^{-1}$ and $\Gamma$, for all $\lambda \in \Lambda$). 

\smallskip

If every closed, normal, non-cocompact subgroup of $G$ has 
finite intersection with $\Lambda$, 
then every infinite normal subgroup $N \normal \Lambda$ contains 
a finite index subgroup of $\Gamma$. (The converse holds as well, and 
is easily verified).

\smallskip

Consequently, if $\varphi: \Lambda \to H$ is a dense homomorphism into a locally compact totally disconnected group $H$,
such that $K:={\overline {\varphi(\Gamma)}}<H$ is compact open and 
$\varphi ^{-1} (K) = \Gamma$, then there is a natural bijection between commensurability 
classes of infinite normal subgroups of $\Lambda$, and commensurability classes of open normal subgroups of $H$.
\end{theorem}

It turns out that a group $H$ as in the second part of the Theorem always exists, and in fact there is a unique minimal one --
the relative profinite completion  $\rpf{\Lambda}{\Gamma}$ of $\Lambda$ w.r.t $\Gamma$, akin to the quotient by a normal subgroup (see \cite{SW09} for details).  Note that it immediately follows
from the Theorem that given any such $H$, every infinite normal subgroup 
of $\Lambda$ 
is of finite index when (and precisely when) $H$ has the same property 
for its own {\it open} subgroups. A good illustration is that of 
$\Gamma = SL_n(\mathbb{Z})< \Lambda = SL_n(\mathbb{Z}[{\frac {1}{p}}])<  
SL_n(\mathbb{R})$. One can take here $H= SL_n(\mathbb{Q} _p)$ to deduce the
well known normal subgroup theorem for $\Lambda$ (the square integrability condition
requires that $n > 2$).

The restriction of $G$ not being a compact extension of abelian 
is necessary. 
Indeed, consider the chain 
$\Gamma = \mathbb{Z}^{n} < \Lambda = \mathrm{SO}(n,\mathbb{Q}) \ltimes (\mathbb{Q}[\sqrt{2}])^{n} < G = \mathrm{SO}(n,\mathbb{R}) \ltimes \mathbb{R}^{n}$.
Here $\Gamma<G$ sits co-compactly (in the abelian part) 
and the latter is ``just infinite'', having only one non-trivial 
normal subgroup -- $\mathbb{R}^{n}$ -- which
is co-compact. Hence the main condition of the Theorem is 
automatically satisfied for all $\Lambda <G$. However, the normal 
subgroup $N = \sqrt{2}~\mathbb{Q}^{n} \normal \Lambda$ (again embedded in
the abelian part of $\Lambda$) intersects $\Gamma$ only trivially.

It is perhaps worth remarking how our main Theorem serves as the natural commensurator analogue 
of the NST for lattices. Recall that in this theory counterpart 
results for homomorphisms of 
dense commensurators yield the same conclusion as those for higher rank lattices, under the crucial assumption that the image of the homomorphism, when restricted to the lattice, is
``large'' in an appropriate sense (depending on the setting in question). Our main theorem is equivalent to the statement that abstract group epimorphisms of $\Lambda$ 
with infinite kernel must have
finite image -- precisely as in the conclusion of the NST -- once their restriction to the lattice is ``large''. This ``largeness'' property should be interpreted as co-finiteness in our entirely abstract setting (see the proof of the Theorem in Section \ref{sec:proofmain}, where this equivalent
formulation is being established and used).

\subsection{Dense Commensurators}

By their nature, rigidity results for homomorphisms of commensurators 
are not as sharp
as those for higher rank lattices, being conditioned on the ``largeness'' of 
the restriction
of the homomorphism to the lattice. However, the advantage of this setting is
that to date 
there are
many more concrete examples of dense commensurators (in well understood groups
$G$, which are essentially simple). One could 
hope 
that our main theorem shall motivate further investigation of the
exotic locally compact groups $H=\rpf{\Lambda}{\Gamma}$ which
naturally arise in this setting. In some cases, this should lead to a full 
normal subgroup 
theorem for the commensurators. The outstanding example here is that when
$G$ is the full automorphism group of a regular tree. Then the full commensurator
of any uniform lattice is known to be dense \cite{Li94}, and in that case its simplicity is still an open question, 
first proposed by Lubotzky-Mozes-Zimmer \cite{LMZ94}.
Exploring the group $\rpf{\Lambda}{\Gamma}$ for such $\Lambda$, or even for 
``smaller'' (but still dense) $\Lambda$, would be of significant interest.
Other exotic settings where density of the commensurator is established 
and to which the Main Theorem applies appear in~\cite{haglund},\cite{remy},\cite{Hag08},\cite{kubena}. We remark that at least the ``amenability half'' of the 
Main Theorem 
applies in 
various known cases of non-uniform lattices with dense commensurators, which
are not finitely generated, as in \cite{abramenko}. 

\subsection{Irreducible Lattices in Products of Groups}

As mentioned earlier, our main theorem yields a sharpening 
of the NST for irreducible lattices in product of groups due to Bader and the 
second author \cite{BS05}, once one of the
ambient factors is totally disconnected. The resulting theorem 
becomes an ``if and only if'' statement:
\begin{corollary}\label{T:bs2}
Let $G$ be a locally compact, second countable compactly generated group, which is not 
a compact extension of an abelian group, and $H$ be any totally disconnected locally compact group.  
Let $\Lambda < G \times H$ be an irreducible (i.e.~has dense projections to the factors) co-compact discrete subgroup. Then each one of the following three conditions
is necessary, and combined together they are sufficient, for the
property that every infinite normal subgroup of $\Lambda$ has finite index:

\begin{enumerate}[{\normalfont (i)}]
\item The intersection of $\Lambda$ with $H$ is finite.
\item The group  $H$ has no infinite index open normal subgroups.
\item The projection of $\Lambda$ to $G$ intersects finitely its closed normal non-cocompact subgroups.
\end{enumerate}

\end{corollary}

Of course, in typical applications of this result the normal subgroup 
structure of $G$ 
and $H$ would be so (simple and) well understood, that the verification of these
conditions is immediate.
A somewhat more technical version of the result when $\Lambda$ is non-uniform
is discussed following its proof, in Section \ref{S:bs2}.
Note that here (unlike \cite{BS05}), 
we do not assume 
that $H$ is compactly generated, and $\Lambda$ may not be finitely generated. 
This variant is relevant e.g.~in the adelic setting, giving a NST for groups of the form
${\bf G}(K)$ where ${\bf G}$ is a simple, simply connected algebraic
group defined over a global field $K$ - see the discussion in Section~\ref{S:bs2}.   
Finally, note that it is easy to turn the above 
result into a ``just infinite'' 
property of $\Lambda$
once the necessary ``no finite normal subgroups'' condition is imposed further on 
$G$ and $H$.

\subsection{On the Approach}

As indicated, the general strategy of the proof of the Main Theorem
follows the original one introduced by Margulis, and consists of two 
entirely independent ``halves'': one pertaining to property (T) and the other 
to amenability, of an appropriate quotient group, which together 
yield finiteness. The property (T) half rests on considerations regarding
reduced cohomology of unitary representations, primarily relying on
 results from \cite{Sha00}, as in the property (T) half of the normal
subgroup theorem for irreducible lattices. Here, however, the 
{\it irreducibility}
of the cohomological representation plays a crucial role, and an additional significant input is 
given by a result of 
Gelander-Karlsson-Margulis \cite{GKM08}.

A key notion in the proof of the amenability half is that of contractive (or 
SAT) group actions. A result of independent interest on which the proof is based is our Contractive Factor Theorem along the 
lines of Margulis' original Factor Theorem (used in the amenability half of 
his NST \cite{Ma79}), and of the Factor Theorem in \cite{BS05}:

\begin{theorem*}[The Contractive Factor Theorem -- Theorem \ref{T:factor}]
Let $G$ be a locally compact second countable group, $\Gamma$ a lattice in $G$, and $\Lambda<G$ a dense subgroup which contains and commensurates $\Gamma$.

Let $(X,\nu)$ be a probability $G$-space such that the restriction of the 
$G$-action to $\Gamma$ is contractive,
and let $(Y,\eta)$ be a $\Lambda$-space such that there exists a $\Gamma$-map $\varphi: (X,\nu) \to (Y,\eta)$.

Then the $\Lambda$-action on $(Y,\eta)$ extends measurably to $G$, in such a 
way that $\varphi$ becomes a $G$-map. More precisely, 
there exists a $G$-space $(Y^{\prime},\eta^{\prime})$, a $G$-map $\varphi^{\prime} : (X,\nu) \to (Y^{\prime},\eta^{\prime})$ and a $\Lambda$-isomorphism $\rho : (Y,\eta) \to (Y^{\prime},\eta^{\prime})$ such that $\varphi^{\prime} = \rho \circ \varphi$ a.e.

\end{theorem*}

Contractive spaces, introduced by Jaworski \cite{Ja94}, \cite{Ja95} under the name strongly approximately transitive (SAT), are the extreme opposite of measure-preserving: $G$ acts on $(X,\nu)$ in such a way that for any measurable set $B$ that is not conull there exists a sequence $\{ g_{n} \}$ in $G$ along which $\nu(g_{n}B) \to 0$.  Jaworski introduced this property to study the Choquet-Deny property on groups and showed that Poisson boundaries are contractive.

The topological analogue of this measurable phenomenon is the following (see Furstenberg and Glasner \cite{FG10}): a continuous action of a group $G$ on a compact metric space $X$ with a quasi-invariant Borel probability measure $\nu$ on $X$ is contractible when for every $x \in X$ there exists $g_{n} \in G$ such that $g_{n}\nu \to \delta_{x}$ weakly.  In \cite{FG10} it is shown that an action 
is measurably contractive if and only if every continuous compact model of it is contractible; it is natural to refer to $G$-spaces for which all models are contractible as contractive and for this reason we adopt the somewhat more descriptive terminology ``contractive''.
It seems that this interesting notion deserves considerably more 
attention; see also the recent \cite{CP12}.

\section{Contractive Actions}

Contractive actions, introduced by Jaworski \cite{Ja94}, \cite{Ja95}, with the idea going back to \cite{Ja91}, have been studied by Kaimanovich \cite{kaimanovichSAT} and by Furstenberg and Glasner \cite{FG10}.  Our aim here is to prove the Contractive Factor Theorem (Theorem \ref{T:factor}) which plays the role of Margulis' Factor Theorem for boundary actions of semisimple groups.

In this paper all probability measure spaces are assumed standard. 
Recall that such a space $(X,\nu)$ is a \textbf{$G$-space} 
when $G$ acts measurably on (a measure one subset of) $X$ such that 
$\nu$ is quasi-invariant under the $G$-action (meaning the measure class is preserved).  A measurable map $\varphi: (X,\nu) \to (Y, \eta)$ between $G$-spaces is a \textbf{$G$-map} when $\varphi$ is $G$-equivariant ($\varphi(gx) = g\varphi(x)$ for all $g \in G$ and $\nu$-almost every $x \in X$), and $\varphi_{*}\nu = \eta$.

\begin{definition}[Jaworski \cite{Ja94}]
A $G$-space $(X,\nu)$ is \textbf{contractive}, also called \textbf{SAT (strongly approximately transitive)}, when for every measurable $B \subseteq X$ with $\nu(B) < 1$ and every $\epsilon > 0$ there exists $g \in G$ such that $\nu(gB) < \epsilon$.
\end{definition}

The primary source of examples of contractive spaces is (quotients of) 
Poisson boundaries, and in the case of measure classes admitting a stationary measure we 
are not familiar with others.
%\begin{theorem}[Jaworksi \cite{Ja94}]
%Let $G$ be a locally compact second countable group and $\mu \in P(G)$ be a Borel probability measure on $G$ such that the support of $\mu$ generates $G$.  Then the Poisson boundary for $(G,\mu)$ is a contractive $G$-space.
%\end{theorem}

\begin{definition}[Furstenberg-Glasner \cite{FG10}]
Let $G$ be a locally compact second countable group acting continuously on a compact metric space $X$, and let $\nu$ be a $G$-quasi-invariant Borel probability measure on $X$.  The action is \textbf{contractible} if 
for every $x \in X$ there exists a sequence $g_{n} \in G$ such that 
$\lim_{n} g_n\nu = \delta_{x}$ weakly  (where $\delta_{x}=$ point mass at $x$).
\end{definition}

\begin{theorem*}[Furstenberg-Glasner \cite{FG10}]
The action on a probability space is contractive if and only if every continuous compact model of it is contractible.
\end{theorem*}

\subsection{Lattices Act Contractively}

\begin{proposition}\label{P:cocptlattcontractive}
Let $G$ be a locally compact second countable group and $\Gamma$ a co-compact 
subgroup of $G$.  Let $(X,\nu)$ be a contractive $G$-space.  Then restricting the action to $\Gamma$ makes $(X,\nu)$ a contractive $\Gamma$-space.
\end{proposition}
\begin{proof}
Let $K \subseteq G$ be a compact set such that $K \Gamma = G$.  Let $B \subseteq X$ such that $\nu(B) < 1$.  Then there exists $g_{n} \in G$ such that $\nu(g_{n}B) \to 0$ since the $G$-action is contractive.  Write $g_{n} = k_{n} \gamma_{n}$ for $k_{n} \in K$ and $\gamma_{n} \in \Gamma$.  Since $K$ is compact there is a convergent subsequence $k_{n_{j}} \to k_{\infty}$.
It is a standard fact (the proof is left to the reader) that if $A_{j}$ is a sequence of measurable sets such that $\nu(A_{j}) \to 0$ and $\ell_{j} \to \ell_{\infty}$ is a convergent sequence then $\nu(\ell_{j}A_{j}) \to 0$.  Applying this to
$A_{j} = g_{n_{j}}B$ and $\ell_{j} = k_{n_{j}}^{-1}$ gives $\nu(\gamma_{n_{j}}B) = \nu(k_{n_{j}}^{-1}g_{n_{j}}B) \to 0$.
\end{proof}

The case of non-uniform lattices turns out to be more difficult, but for our purposes it will suffice to establish the result for actions on (quotients of) the 
Poisson boundary of the ambient group.  However, the first author and J.~Peterson \cite{CP12} have, subsequent to our work, shown that the existence of a stationary measure in the class suffices: if $\Gamma <G$ is a lattice, and $G$ acts contractively on a stationary $G$-space, then so does $\Gamma$. The general case remains conjectured.

\begin{proposition}\label{P:latticecontractive}
Let $G$ be a locally compact second countable group and $\Gamma$ a lattice in $G$.  Then the action of $\Gamma$ on any (quotient of the) Poisson boundary of $G$ relative to a symmetric measure with support generating $G$ is contractive.
\end{proposition}
\begin{proof}
Let $(X,\nu)$ be any compact model of (a quotient of) the Poisson boundary for $(G,\mu)$, where $\mu$ is a symmetric Borel probability measure on $G$ with support generating $G$.  The key feature of the Poisson boundary of use to us is that for any 
$L^{\infty}$-function $f$ on $X$ (and in particular for any continuous $f$), 
the maps $\varphi_{n} : G^{\mathbb{N}} \to \mathbb{R}$ defined by $\varphi_{n}(\omega_{1},\omega_{2},\ldots) = \int f(\omega_{1}\cdots\omega_{n} x)~d\nu(x)$ form a martingale (by the stationarity of $\nu$) and hence converge $\mu^{\mathbb{N}}$-almost surely.  Therefore, for $\mu^{\mathbb{N}}$-almost every $\omega \in G^{\mathbb{N}}$ the measures $\omega_{1}\cdots\omega_{n}\nu$ converge weakly to some limit measure $\nu_{\omega}$.  In fact, $\nu_{\omega}$ is a point mass almost surely (which characterizes quotients of the Poisson boundary), 
 thus showing that a boundary action is contractive.  The stationarity of $\nu$ implies that for any measurable set $A$, $\int \nu_{\omega}(A)~d\mu^{\mathbb{N}}(\omega) = \nu(A)$, and therefore $\mu^{\mathbb{N}} \{ \omega : \nu_{\omega}(A) = 1 \} = \nu(A)$.

Let $m$ be the Haar measure on $G / \Gamma$ normalized to be a probability measure.  Recall that Kakutani's Random Ergodic Theorem \cite{kakutani} % in the measure-preserving case 
%and to Kifer \cite{kifer} in the general case; see \cite{Fur00} for details)
states that for any $f \in L^{\infty}(G / \Gamma)$, $m$-almost every $z \in G / \Gamma$ and $\mu^{\mathbb{N}}$-almost every $\omega \in G^{\mathbb{N}}$,
\[
\lim_{N \to \infty} \frac{1}{N} \sum_{n=1}^{N} f(\omega_{n}\cdots\omega_{1}z) = \int f d\nu.
\]
Let $K_{0}$ be any open bounded subset of $G$. 
%such that $m(K_{0} \cap F) > 0$ where $F$ is a fixed fundamental domain for $\Gamma$.  
Write $K$ for the finite (positive) measure subset of $G / \Gamma$ 
which is the 
image of $K_{0}$ under the quotient map. Let $\mathbbm{1}_{K}$ be the characteristic function of $K$.  By the Random Ergodic Theorem, for $m$-almost every $z$ and $\mu^{\mathbb{N}}$-almost every $\omega \in G^{\mathbb{N}}$,
\[
\lim_{N\to\infty} \frac{1}{N}\sum_{n=1}^{N} \mathbbm{1}_{K}(\omega_{n}\cdots\omega_{1} z) = m(K) > 0.
\]
Pick $z \in G / \Gamma$ such that the above holds for $\mu^{\mathbb{N}}$-almost every $\omega$.  Then $\omega_{n}\cdots \omega_{1} z \in K$ infinitely often $\mu^{\mathbb{N}}$-almost surely.  As $\mu$ is symmetric, also $\omega_{n}^{-1}\cdots\omega_{1}^{-1}z \in K$ infinitely often $\mu^{\mathbb{N}}$-almost surely.

Fix an arbitrary measurable set $B \subseteq X$ with $\nu(B) < 1$.  Let $z_{0} \in G$ be a representative of $z \in G / \Gamma$.  Set $A = z_{0}B$ and note that $\nu(A) < 1$ (by quasi-invariance).  As recalled above,
\[
\mu^{\mathbb{N}} \{ \omega : \nu_{\omega}(A) = 0 \} = 1 - \nu(A) > 0.
\]
Pick $\omega$ such that $\nu_{\omega}(A) = 0$ and $\omega_{n}^{-1} \cdots \omega_{1}^{-1}z \in K$ infinitely often 
(the intersection of a positive measure set with a full measure set is nonempty).  Let $\{ n_{j} \}$ be the times such that $\omega_{n_{j}}^{-1}\cdots\omega_{1}^{-1}z \in K$ (which is happening in $G / \Gamma$).  Then
\[
0 = \nu_{\omega}(A) = \lim_{n} \omega_{1}\cdots\omega_{n}\nu(A) = \lim_{j\to\infty} \nu(\omega_{n_{j}}^{-1}\cdots\omega_{1}^{-1}A) = \lim_{j\to\infty} \nu(\omega_{n_{j}}^{-1}\cdots\omega_{1}^{-1}z_{0}B)
\]
and $\omega_{n_{j}}^{-1}\cdots\omega_{1}^{-1}z_{0} \in K_{0}\Gamma$ for each $j$ (since $z_{0} \in z\Gamma$ and $\omega_{n_{j}}^{-1}\cdots\omega_{1}^{-1}z \in K$).

Write $\omega_{n_{j}}^{-1}\cdots\omega_{1}^{-1}z_{0} = k_{j}\gamma_{j}$ for $k_{j} \in K_{0}$ and $\gamma_{j} \in \Gamma$.  Then
$\lim_{j\to\infty} \nu(k_{j} \gamma_{j} B) = 0$.
Choose a subsequence $j_{\ell}$ such that $k_{j_{\ell}} \to k_{\infty}$ for some $k_{\infty} \in K$ and set $B_{\ell} = k_{j_{\ell}} \gamma_{j_{\ell}} B$.  Then $\nu(B_{\ell}) \to 0$ and $k_{j_{\ell}}^{-1} \to k_{\infty}^{-1}$, so, 
%as in the proof of the previous Proposition, 
$\nu(k_{j_{\ell}}^{-1} B_{\ell}) \to 0$ (an easy exercise on the continuity
of the $G$-action on $L^1$).  Hence
$\lim_{\ell \to \infty} \nu(\gamma_{j_{\ell}} B) = 0$, and the $\Gamma$-action is contractive.
\end{proof}

\subsection{Uniqueness of Quotients of Contractive Spaces}

We begin by observing the following simple basic fact:
\begin{lemma}\label{L:contractivemeasureclass}
Let $(X,\nu)$ be a $G$-space and $\nu^{\prime}$ a Borel probability measure on $X$ in the same measure class as $\nu$.  Let $\{ g_{n} \}$ be a sequence in $G$ and $B \subseteq X$ a measurable set.  If $\nu(g_{n}B) \to 0$ then $\nu^{\prime}(g_{n}B) \to 0$.  In particular, if $(X,\nu)$ is contractive then so is $(X,\nu^{\prime})$.
\end{lemma}
\begin{proof}
Suppose that $\limsup \nu^{\prime}(g_{n}B) = \delta > 0$.  Let $\{ n_{j} \}$ be the sequence attaining this limit.  Then $\nu(g_{n_{j}}B) \to 0$ and $\nu^{\prime}(g_{n_{j}}B) \to \delta$.  Pick a further subsequence $\{ n_{j_{t}} \}$ such that $\nu(g_{n_{j_{t}}} B) < 2^{-t}$.  Define
$B_{k} = \bigcup_{t=k}^{\infty} g_{n_{j_{t}}} B$
and observe that
$\nu(B_{k}) \leq \sum_{t=k}^{\infty} \nu(g_{n_{j_{t}}}B) \leq \sum_{t=k}^{\infty} 2^{-t} = 2^{-k+1} \to 0$
but
$\nu^{\prime}(B_{k}) \geq \nu^{\prime}(g_{n_{j_{k}}} B) \to \delta$.
As the $B_{k}$ are decreasing,
$\nu(\bigcap_{k} B_{k}) = 0$ but $\nu^{\prime}(\bigcap_{k} B_{k}) \geq \delta$
contradicting that the measures are in the same class.
\end{proof}

We will need a basic fact about the existence of compact models.  This result does not seem to appear explicitly in the literature, but the proof is essentially contained in \cite{Zi84} and follows from \cite{vara}.

\begin{lemma}[Varadarajan \cite{vara}]\label{L:compactmodels}
Let $G$ be a locally compact second countable group, let $(X,\nu)$ and $(Y,\eta)$ be $G$-spaces, let $\nu^{\prime}$ be a Borel probability measure in the same measure class as $\nu$, let $\eta^{\prime} \in P(Y)$ be in the same measure class as $\eta$ and let $\pi : (X,\nu) \to (Y,\eta)$ and $\pi^{\prime} : (X,\nu^{\prime}) \to (Y,\eta^{\prime})$ be $G$-maps.

Then there exists a continuous compact model for $\pi$ and $\pi^{\prime}$, that is, there exist compact metric spaces $X_{0}$ and $Y_{0}$ on which $G$ acts continuously, fully supported Borel probability measures $\nu_{0}, \nu_{0}^{\prime}$ on $X$ and $\eta_{0}, \eta_{0}^{\prime}$ on $Y$, continuous $G$-equivariant maps $\pi_{0} : X_{0} \to Y_{0}$ and $\pi_{0}^{\prime} : X_{0} \to Y_{0}$ and measurable $G$-isomorphisms $\Phi : (X,\nu) \to (X_{0},\nu_{0})$ and $\Psi : (Y,\eta) \to (Y_{0},\eta_{0})$ (which are also $G$-isomorphisms $\Phi : (X,\nu^{\prime}) \to (X_{0},\nu_{0}^{\prime})$ and $\Psi : (Y,\eta^{\prime}) \to (Y_{0},\eta_{0}^{\prime})$) such that the resulting diagrams commute: $\Psi^{-1} \circ \pi_{0} \circ \Phi = \pi$ and $(\Psi^{\prime})^{-1} \circ \pi_{0}^{\prime} \circ \Phi = \pi^{\prime}$.
\end{lemma}
\begin{proof}
Let $\mathcal{X}$ be a countable collection of functions in $L^{\infty}(X,\nu) = L^{\infty}(X,\nu^{\prime})$ that separates points and let $\mathcal{Y}$ be a countable collection in $L^{\infty}(Y,\eta)$ that separates points.  Let $\mathcal{F} = \mathcal{X} \cup \{ f \circ \pi : f \in \mathcal{Y} \} \cup \{ f \circ \pi^{\prime} : f \in \mathcal{Y} \}$.  Then $\mathcal{F}$ is a countable collection.  Let $B$ be the unit ball in $L^{\infty}(G,\Haar)$ which is a compact metric space in the weak-* topology.

Define $X_{00} = \prod_{f \in \mathcal{F}} B$ and $Y_{00} = \prod_{f \in \mathcal{Y}} B$, both of which are compact metric spaces using the product topology.  Define $\pi_{00} : X_{00} \to Y_{00}$ to be the restriction map: for $f \in \mathcal{Y}$ take the $f^{th}$ coordinate of $\pi_{00}(x_{00})$ to be the $(f \circ \pi)^{th}$ coordinate of $x_{00}$.  Then $\pi_{00}$ is continuous.  Likewise, define $\pi_{00}^{\prime} : X_{00} \to Y_{00}$ by setting the $f^{th}$ coordinate of $\pi_{00}^{\prime}(x_{00})$ to be the $(f \circ \pi^{\prime})^{th}$ coordinate of $x_{00}$.

Define the map $\Phi : X \to X_{00}$ by $\Phi(x) = (\varphi_{f}(x))_{f \in \mathcal{F}}$ where $(\varphi_{f}(x))(g) = f(gx)$.  Then $\Phi$ is an injective map (since $\mathcal{F}$ separate points).  Observe that $(\varphi_{f}(hx))(g) = f(ghx) = (\varphi_{f}(x))(gh)$.  Consider the $G$-action on $X_{00}$ given by the right action on each coordinate.  Then $G$ acts on $X_{00}$ continuously (and likewise on $Y_{00}$ continuously) and $\Phi$ is $G$-equivariant.  Similarly, define $\Psi : Y \to Y_{00}$ by $\Psi(y) = (\psi_{f}(y))_{f \in \mathcal{Y}}$ where $(\psi_{f}(y))(g) = f(gy)$.

Let $X_{0} = \overline{\Phi(X)}$, let $\nu_{0} = \Phi_{*}\nu$, let $Y_{0} = \overline{\Psi(Y)}$, let $\eta_{0} = \Psi_{*}\eta$ and let $\pi_{0}$ be the restriction of $\pi_{00}$ to $X_{0}$.  Then $\Phi : (X,\nu) \to (X_{0},\nu_{0})$ and $\Psi : (Y,\eta) \to (Y_{0},\eta_{0})$ are $G$-isomorphisms.  Since $(\psi_{f}(\pi(x)))(g) = f(g\pi(x)) = f \circ \pi (gx) = (\varphi_{f \circ \pi}(x))(g)$, $\pi_{0}(X_{0}) = Y_{0}$ and $\Psi^{-1} \circ \pi_{0} \circ \Phi = \pi$.  Likewise, letting $\nu_{0}^{\prime} = \Phi_{*}\nu^{\prime}$, $\eta_{0}^{\prime} = \Psi_{*}\eta^{\prime}$, and $\pi_{0}^{\prime}$ be the restriction of $\pi_{00}^{\prime}$ to $X_{0}$, one readily observes that $\Psi^{-1} \circ \pi_{0}^{\prime} \circ \Phi = \pi^{\prime}$.
\end{proof}

\begin{remark}
The previous result extends to a countable collection of $G$-maps $\pi_{n} : (X,\nu_{n}) \to (Y,\eta_{n})$ when the $\nu_{n}$ are all in the same measure class.
\end{remark}

\begin{proposition}\label{P:contractivemapunique}
Let $G$ be a locally compact second countable group.  Let $(X,\nu)$ be a contractive $G$-space and $(Y,\eta)$ be a $G$-space.  Let $\varphi : (X,\nu) \to (Y,\eta)$ and $\varphi^{\prime} : (X,\nu) \to (Y,\eta^{\prime})$ be $G$-maps such that $\eta$ and $\eta^{\prime}$ are in the same measure class.  Then $\varphi = \varphi^{\prime}$ almost everywhere.
\end{proposition}
\begin{proof}
By Lemma \ref{L:compactmodels}, take $X$ and $Y$ to be compact metric spaces where $G$ acts continuously and such that $\varphi,\varphi^{\prime} : X \to Y$ are continuous maps.
Since $(X,\nu)$ is a contractive $G$-space, the model is contractible.  Let $x_{0} \in X$ be arbitrary.  Then there exists a sequence $g_{n} \in G$ such that $g_{n}\nu \to \delta_{x_{0}}$ weakly.

Since $\varphi$ is continuous so is the pushforward map $\varphi_{*}$ and therefore $\varphi_{*}(g_{n}\nu) \to \varphi_{*}(\delta_{x_{0}})$.  By the $G$-equivariance of $\varphi$ this means $g_{n}\eta = g_{n} (\varphi_{*}\nu) \to \varphi_{*}(\delta_{x_{0}}) = \delta_{\varphi(x_{0})}$.
Of course the same reasoning gives that $g_{n}\eta^{\prime} \to \delta_{\varphi^{\prime}(x_{0})}$.  

Let $B \subseteq Y$ be any open set containing $\varphi(x_{0})$.  Then $g_{n}\eta(B^{C}) \to \delta_{\varphi(x_{0})}(B^{C}) = 0$ since $B^{C}$ is a continuity set for $\delta_{\varphi(x_{0})}$ (the portmanteau Theorem).  By Lemma \ref{L:contractivemeasureclass}, $g_{n}\eta^{\prime}(B^{C}) \to 0$ also so $\varphi^{\prime}(x_{0}) \in B$.  As this holds for all open sets $B$ 
containing $\varphi(x_{0})$, it follows that $\varphi^{\prime}(x_{0}) = \varphi(x_{0})$.  Since $x_{0}$ was arbitrary this means that $\varphi = \varphi^{\prime}$ as maps between the compact models.  So $\varphi = \varphi^{\prime}$ measurably.
\end{proof}

\subsection{The Contractive Factor Theorem}

\begin{theorem}\label{T:factor}
Let $G$ be a locally compact second countable group, $\Gamma<G$ a lattice, 
and $\Lambda<G$ a dense subgroup that contains and commensurates $\Gamma$.

Let $(X,\nu)$ be a $G$-space such that the restriction of the action to 
$\Gamma$ is contractive,
and let $(Y,\eta)$ be a $\Lambda$-space such that there exists a $\Gamma$-map $\varphi: (X,\nu) \to (Y,\eta)$.

Then the $\Lambda$-action on $(Y,\eta)$ extends measurably to $G$, in such a 
way that $\varphi$ becomes a $G$-map. More precisely, 
there exists a $G$-space $(Y^{\prime},\eta^{\prime})$, a $G$-map $\varphi^{\prime} : (X,\nu) \to (Y^{\prime},\eta^{\prime})$ and a $\Lambda$-isomorphism $\rho : (Y,\eta) \to (Y^{\prime},\eta^{\prime})$ such that $\varphi^{\prime} = \rho \circ \varphi$ a.e.

%Then $\varphi$ is a $\Lambda$-map and it extends to a $G$-map: there exists a $G$-space $(Y^{\prime},\eta^{\prime})$, a $G$-map $\varphi^{\prime} : (X,\nu) \to (Y^{\prime},\eta^{\prime})$ and a $\Lambda$-isomorphism $\rho : (Y,\eta) \to (Y^{\prime},\eta^{\prime})$ such that $\varphi^{\prime} = \rho \circ \varphi$ a.e.

\end{theorem}
\begin{proof}
Fix $\lambda \in \Lambda$.  Since $\Lambda$ commensurates $\Gamma$, the subgroup $\Gamma_{0} = \Gamma \cap \lambda^{-1}\Gamma\lambda$ is also a lattice in $G$.
Consider the map $\varphi_{\lambda} : X \to Y$ given by
\[
\varphi_{\lambda}(x) := \lambda^{-1} \varphi(\lambda x)
\]
Since $\varphi$ is $\Gamma$-equivariant, $\varphi_{\lambda}$ is $\Gamma_{0}$-equivariant: for $\gamma_{0} \in \Gamma_{0}$ we have
$\lambda \gamma_{0} \lambda^{-1} \in \Gamma$
and so
\[
\varphi_{\lambda}(\gamma_{0} x) = \lambda^{-1}\varphi(\lambda \gamma_{0}x) =  \lambda^{-1}\varphi((\lambda \gamma_{0} \lambda^{-1}) \lambda x)
= \lambda^{-1} (\lambda \gamma_{0} \lambda^{-1}) \varphi(\lambda x) = \gamma_{0} \varphi_{\lambda}(x).
\]

Let $\eta = \varphi_{*}\nu$ be the pushforward of $\nu$ to $Y$ over $\varphi$ and $\eta^{\prime} = (\varphi_{\lambda})_{*}\nu$ be the pushforward over $\varphi_{\lambda}$.  Then $\eta$ and $\eta^{\prime}$ are in the same measure class: if $\eta(A) = 0$ then $\eta(\lambda A) = 0$ by the $\Lambda$-quasi-invariance of $\eta$, and therefore $\nu(\varphi^{-1}(\lambda A)) = 0$. But
$\eta^{\prime}(A) = \lambda\nu(\varphi^{-1}(\lambda A))$
so by the $\Lambda$-quasi-invariance of $\nu$ this is zero, hence the measures are in the same class.

By Proposition \ref{P:cocptlattcontractive}, the action of $\Gamma_{0}$ on $(X,\nu)$ is contractive since the $\Gamma$-action is. 
% Then the action of $\Gamma_{0}$ is contractive on $(Y,\eta)$ since it is a quotient of a contractive space.
Since $\varphi$ and $\varphi_{\lambda}$ are both $\Gamma_{0}$-equivariant maps, one relative to a contractive $\Gamma_{0}$-space, and one relative to another measure in the class of the contractive measure, by Proposition \ref{P:contractivemapunique}, $\varphi_{\lambda} = \varphi$ a.e.  Hence for each $\lambda$ we have that
$\lambda^{-1}\varphi(\lambda x) = \varphi(x)$
for almost every $x$, making $\varphi$ a $\Lambda$-map. 
%(when $\Lambda$ is not countable, by the second countability of $G$, we can work with a countable subset of $\Lambda$ which is dense in $\Lambda$ in the relative topology).

Treating $L^{\infty}(Y,\eta)$ as a $\Lambda$-invariant sub-$\sigma$-algebra of the $G$-invariant $\sigma$-algebra $L^{\infty}(X,\nu)$, the density of $\Lambda$ in $G$ means that as a $\sigma$-algebra $L^{\infty}(Y,\eta)$ is $G$-invariant.  Then by Mackey's point realization there exists a $G$-space $(Y^{\prime},\eta^{\prime})$ measurably $\Lambda$-isomorphic to $(Y,\eta)$, and a $G$-map $(X,\nu) \to (Y^{\prime},\eta^{\prime})$ such that this map composed with the $\Lambda$-isomorphism is $\varphi$.
\end{proof}

\section{Proof of the Main Theorem}\label{sec:proofmain}

\begin{theorem}\label{T:1}
Let $G$ be a locally compact, second countable, compactly generated group that is not a compact extension of an abelian group.

Let $\Gamma < G$ be a finitely generated square integrable lattice and let $\Lambda < G$ be a dense subgroup of $G$ that contains and commensurates $\Gamma$.

Then every infinite normal subgroup $N \normal \Lambda$ has the property that $N \cap \Gamma$ has finite index in $\Gamma$, if and only if $\Lambda$ intersects finitely every closed normal non-cocompact subgroup of $G$.
\end{theorem}

Theorem \ref{T:1} will be a consequence of the following Proposition.  
We shall 
first state it and prove that Theorem \ref{T:1} follows from it, and then turn to proving the result.

\begin{proposition}[The Reduction Step]\label{P:2}
Let $\Gamma <_{c} \Lambda < G$ be as in Theorem \ref{T:1}, but with no structural restriction on $G$. Let $N$ be a normal subgroup 
of $\Lambda$ such that
$\Gamma$ maps onto $\Lambda / N$ via the coset map $\Lambda \to \Lambda / N$,
and $\overline{[N,N]}$ is co-compact in $G$ (hence $\overline {N}$ is as well).
Then $\Lambda / N$ is finite.
\end{proposition}

\begin{proof}[Proof of Theorem \ref{T:1} assuming Proposition \ref{P:2}]
Assume first that $\Lambda$ intersects finitely every closed normal non-cocompact subgroup of $G$, and let $N \normal \Lambda$ be any infinite normal subgroup. 
Then $\overline{N} \normal \overline{\Lambda} = G$ and since $N$ is infinite and contained in $\Lambda \cap \overline{N}$ it follows from the assumption of the Theorem
that $\overline{N}$ is co-compact in $G$. Now 
$[N,N]$ is a characteristic normal subgroup of $N$, hence also $[N,N] \normal \Lambda$. 
%and $\overline{[N,N]} \normal \overline{\Lambda} = G$.% 
 Then either $[N,N]$ is finite, or it's infinite, in which case the exact same argument as before (with $N$ replaced by $[N,N]$) shows
that  $\overline{[N,N]}$ is co-compact in $G$.  We now observe that the first possibility cannot occur.

%Suppose that $[N,N]} is finite.  Then $[N,N] \subseteq \overline{[N,N]} \cap \Lambda$ is finite.  Let $h_{1},h_{2} \in \overline{N}$.  Then $h_{1} = \lim n_{1,m}$ and $h_{2} = \lim n_{2,m}$ for some $n_{1,m},n_{2,m} \in N$ and so $h_{1}h_{2}h_{1}^{-1}h_{2}^{-1} = \lim n_{1,m}n_{2,m}n_{1,m}^{-1}n_{2,m}^{-1}$.
%Since $n_{1,m}n_{2,m}n_{1,m}^{-1}n_{2,m}^{-1} \in [N,N]$, which is finite, there is a subsequence along which $n_{1,m}n_{2,m}n_{1,m}^{-1}n_{2,m}^{-1} = n \in [N,N]$ is constant.  Then $h_{1}h_{2}h_{1}^{-1}h_{2}^{-1} = n \in [N,N]$ and we conclude that $[\overline{N},\overline{N}] = [N,N]$ is finite.

Indeed, it is a general fact that $\overline {[N,N]} = \overline { [\overline {N}, \overline {N}]}$, hence the assumed finiteness property of  $[N,N]$ implies that property
for the left, hence also for the right hand side.  
Now $\overline{N}<G$ is co-compact so it inherits compact generation from $G$.  By the general Lemma \ref{L:finindtrick} below it then follows from this finiteness property that the center $Z(\overline{N})$ has 
finite index in $\overline{N}$, hence is co-compact as well in $G$.  Being a characteristic normal subgroup of $\overline{N}$, it is also normal in $G$.  
Hence $G$ is a compact extension of the abelian group $Z(\overline{N})$, contradicting the hypothesis of the Theorem.  
We conclude that the second possibility holds: $\overline{[N,N]}$ is co-compact in $G$.

Let $\Lambda^{\prime} = \Gamma \cdot N$.  Then $\Lambda^{\prime}$ is a subgroup of $\Lambda$ that contains and commensurates $\Gamma$.  
Clearly $\Gamma$ maps onto $\Lambda^{\prime} / N$ via the coset map $\gamma \mapsto \gamma N$.
We are now in position to apply Proposition \ref{P:2} to the groups $\Gamma < \Lambda^{\prime} < \overline{\Lambda^{\prime}}$ with $N \normal \Lambda^{\prime}$ (the closure of 
$[N,N]$, being co-compact in $G$, is so in $\overline{\Lambda^{\prime}}$ as well). It follows from this Proposition 
that $\Lambda^{\prime} / N$ is finite.  Then $\Gamma / (\Gamma \cap N) \simeq (\Gamma \cdot N) / N$ is finite as well, so $N$ contains a finite index subgroup of $\Gamma$, as required.

The reverse direction of Theorem \ref{T:1} is easy, and we prove it for completeness.
 Assume that for every infinite normal subgroup $N \normal \Lambda$ it holds that $N \cap \Gamma$ has finite index in $\Gamma$.  
We need to show that every closed $M \normal G$ which intersects $\Lambda$ infinitely, is co-compact.  
Given such $M$, set $N = M \cap \Lambda \normal \Lambda $, noting that 
here by the reverse assumption of the Theorem $N \cap \Gamma = M \cap \Gamma$ 
has finite index in $\Gamma$.  Since (every finite index subgroup of) $\Gamma$ has co-finite Haar measure in $G$, it follows that so does the normal subgroup $M \normal G$. 
Hence the group $G / M$ has finite Haar measure, and is therefore compact, as required. This completes the reduction of the proof of Theorem \ref{T:1} to
 Proposition \ref{P:2}, modulo the following general (and probably well known) Lemma.
\end{proof}

\begin{lemma}\label{L:finindtrick}
Let $H$ be a compactly generated, second countable locally compact group, for which $[H,H]$ is finite.  Then the center $Z(H)$ has finite index in $H$.
\end{lemma}
\begin{proof}
Let $K \subseteq H$ be a compact generating set.  For $x \in K$ consider the orbit of $x$ under conjugation by $H$: $h \mapsto hxh^{-1}$.  Since $[H,H]$ is finite, $hxh^{-1}x^{-1}$ takes on only finitely many values, so for each $x$, the orbit $\{ hxh^{-1} : h \in H \}$ is finite.  Therefore $H_{x} = \{ h \in H : hxh^{-1} = x \}$ has finite index in $H$.

Each $H_{x}$ is compactly generated since $H$ is.  Let $Q_{x} \subseteq H_{x}$ be a compact generating set. For $q \in Q_{x}$ observe that $qxq^{-1}x^{-1} = e$.  By the continuity of the action of $H$ on itself there is then an open neighborhood $U_{x}$ of $x$ such that $qyq^{-1}y^{-1} = e$ for all $q \in Q_{x}$ and all $y \in U_{x}$.  This can be seen as follows: if no such neighborhood exists then there exists $x_{n} \to x$ and $q_{n} \in Q_x$ such that $q_{n}x_{n}q_{n}^{-1}x_{n}^{-1} \ne e$.  Since $q_{n}x_{n}q_{n}^{-1}x_{n}^{-1} \in [H,H]$ is a finite set there is a subsequence on which $q_{n}x_{n}q_{n}^{-1}x_{n}^{-1} = z \ne e$ is constant.  Take a further subsequence along which $q_{n} \to q \in Q_{x}$ (compactness of $Q_{x}$).  Then $q_{n}x_{n}q_{n}^{-1}x_{n}^{-1} \to qxq^{-1}x^{-1}$ and $q_{n}x_{n}q_{n}^{-1}x_{n}^{-1} = z$ hence $qxq^{-1}x^{-1} = z \ne e$ contradicting that $q \in H_{x}$.

Therefore, for all $x \in K$ there is an open neighborhood $U_{x}$ of $x$ such that for all $q \in Q_{x}$ and all $y \in U_{x}$ we have $qyq^{-1}y^{-1} = e$.  Since $Q_{x}$ generates $H_{x}$ this means that $U_{x}$ commutes with $H_{x}$.  Now $K \subseteq \bigcup_{x \in K} U_{x}$ is an open cover of a compact set hence there is a finite subcover: $K \subseteq \bigcup_{j=1}^{\ell} U_{x_{j}}$ for some $x_{1},\ldots,x_{\ell} \in K$.  Let
$H_{0} = \bigcap_{j=1}^{\ell} H_{x_{j}}$.
Then $H_{0}$ commutes with $U_{x_{1}},\ldots,U_{x_{\ell}}$ hence it commutes with $K$ and therefore $H_{0}$ commutes with all of $H$.  Now $H_{0}$ has finite index in $H$ since it is a finite intersection of finite index subgroups of it, hence $H_{0} \subseteq Z(H)$ and the latter has finite index, as claimed.
\end{proof}

In the rest of this section we prove Proposition \ref{P:2}. This is done 
in two independent parts: the ``amenability half'' and the ``property $(T)$ 
half'', 
which are Propositions \ref{P:amenhalf} and \ref{P:Thalf} below. 

\begin{proof}[Proof of Proposition \ref{P:2} from Propositions \ref{P:amenhalf} and \ref{P:Thalf} below] The group
$\Lambda / N$ has property $(T)$ by Proposition \ref{P:Thalf} and is amenable by Proposition \ref{P:amenhalf}, hence it is finite.
\end{proof}

\subsection{The Amenability Half}

\begin{proposition}\label{P:amenhalf}
Let $G$ be a locally compact second countable group and let $\Gamma < G$ be a lattice in $G$.  Let $\Lambda < G$ be a dense subgroup that contains and commensurates $\Gamma$.

Let $N$ be a normal subgroup of $\Lambda$ such that $\overline{N}$ is co-compact in $G$, and such that
 $\Gamma$ maps onto $\Lambda / N$ via the coset map.  Then $\Lambda / N$ is amenable.
 \end{proposition}
 \begin{proof}
Since $\Lambda / N$ is (second) countable, it is amenable if for any 
compact {\it metric} space on which $\Lambda / N$ acts continuously, there is a $\Lambda / N$ invariant probability measure.
Let $Z$ be such a space, viewed as a  $\Lambda$-space with trivial action of 
$N$.

Let $(X,\nu)$ be the Poisson boundary of $G$ (with respect to any symmetric measure with support generating $G$).
By Proposition \ref{P:latticecontractive}, the action of $\Gamma$ on $(X,\nu)$ is contractive.
The $G$-action on $(X,\nu)$ is amenable, hence also that 
of its closed subgroup $\Gamma$ \cite{Zi84}.
Let then $\varphi : X \to P(Z)$ be a measurable $\Gamma$-equivariant map .  Let $Y = P(Z)$ and $\eta = \varphi_{*}\nu \in P(Y)$ so that $\varphi: (X,\nu) \to (Y, \eta)$ is a $\Gamma$-map.

By hypothesis, $\Gamma$ maps onto $\Lambda / N$ via the coset map $\gamma \mapsto \gamma N$ so for any $\lambda \in \Lambda$ there is some $\gamma \in \Gamma$ such that $\gamma N = \lambda N$.  Since $N$ acts trivially on $Z$, we have $\lambda \eta = \gamma \eta$ and therefore the $\Gamma$-quasi-invariance of $\eta$ implies $\Lambda$-quasi-invariance, so $(Y,\eta)$ is a $\Lambda$-space.

By the Contractive Factor Theorem (Theorem \ref{T:factor}), $\varphi$ extends to a $G$-map to a $\Lambda$-isomorphic $G$-space $(Y^{\prime},\eta^{\prime})$.  Since $N$ acts trivially on $Z$ the same is true on $Y = P(Z)$ and therefore $\overline{N}$ acts trivially on $Y^{\prime}$.  As $\eta$ is invariant under $N$, $\eta^{\prime}$ is $\overline{N}$-invariant.

Let $Q = G / \overline{N}$.  Then $Q$ is a compact group.  Since $\eta^{\prime}$ is quasi-invariant under $G$ it also is under $Q$.  Let $m$ be the Haar measure on $Q$ normalized to be a probability measure and set $\eta^{\prime\prime} = m * \eta^{\prime}$.  Then $\eta^{\prime\prime}$ is in the same measure class as $\eta^{\prime}$, and $\eta^{\prime\prime}$ is $Q$-invariant.  Therefore $\eta^{\prime\prime}$ is $G$-invariant since $\overline{N} \normal G$ and $\eta^{\prime}$ is $\overline{N}$-invariant.

Let $\eta^{\prime\prime\prime}$ be the isomorphic image of $\eta^{\prime\prime}$ on $Y$.  So $\eta^{\prime\prime\prime}$ is a $\Lambda$-invariant probability measure on $Y = P(Z)$.  Take $\rho$ to be the barycenter of $\eta^{\prime\prime\prime}$: $\rho = \int_{P(Z)} \zeta~d\eta^{\prime\prime\prime}(\zeta)$.
Then $\rho \in P(Z)$ is $\Lambda$-invariant since
$\lambda \rho = \int_{P(Z)} \lambda\zeta~d\eta^{\prime\prime\prime}(\zeta)
= \int_{P(Z)} \zeta~d\lambda\eta^{\prime\prime\prime}(\zeta)
= \rho$.  Hence $\rho$ is a $\Lambda / N$-invariant probability measure on $Z$
and the proof is complete.
\end{proof}

\subsection[The Property (T) Half]{The Property $(T)$ Half}\label{S:Thalf}

The requirement that $\Gamma$ be square-integrable in the main Theorem is only necessary for the property $(T)$ half of the proof. Recall that if $\Gamma$ is a finitely generated lattice in a locally compact group $G$ then $\Gamma$ is {\bf square integrable} when there exists a fundamental domain $F$ for $G / \Gamma$ such that
\[
\int_{F} | \alpha(g,x) |^{2}~dm(x) < \infty
\]
where $\alpha : G \times F \to \Gamma$ is the cocycle given by $\alpha(g,x) = \gamma$ if and only if $gx\gamma \in F$, $| \cdot |$ denotes the word length in $\Gamma$ (the choice of generating set will not affect the finiteness of the integral), and $m$ is the finite Haar measure on $F$.  
This requirement is crucial in order to be able to define a natural ($L^2$-)induction 
map (going from $\Gamma$ to $G$) on the first cohomology with unitary coefficients, and is imposed in order to utilize the rigidity results of \cite{Sha00}
and \cite{GKM08}. 
Of course, all uniform lattices are square integrable. Non-uniform lattices are known to be square integrable in higher-rank semisimple groups \cite{Sha00}, rank-one simple Lie groups not locally isomorphic to $SL_{2}(\mathbb{R})$ or $SL_{2}(\mathbb{C})$ \cite{Sha00b}, and in the Kac-Moody case \cite{remy3}.

Before stating the property $(T)$ half, we derive a consequence of a
result of Gelander-Karlsson-Margulis \cite{GKM08}:
\begin{proposition}\label{P:GKM}
Let $G$ be a locally compact second countable group and let $\Gamma < G$ be a square-integrable lattice in $G$.  Let $\Lambda < G$ be a dense subgroup that contains and commensurates $\Gamma$.  Let $\pi : \Lambda \to \mathcal{H}$ be an
 irreducible unitary representation of $\Lambda$ on a Hilbert space 
$\mathcal{H}$ 
such that every finite index subgroup of $\Gamma$ does not admit almost 
invariant vectors for $\pi$.  Then any affine action of $\Lambda$ on $\mathcal{H}$ with no fixed points, whose linear part is given by $\pi$, 
extends to a continuous affine isometric $G$-action on $\mathcal{H}$.
\end{proposition}
\begin{proof}
We want to apply Theorem 8.1 in \cite{GKM08} (and Remark 8.9 regarding square integrable lattices) to the $\Lambda$-action on $\mathcal{H}$. It is easy to check in complete generality that for any 
affine isometric action of a finitely generated group $\Gamma_0$ on a Hilbert
space $\mathcal{H}$, the displacement function relative to a generating set
of $\Gamma_0$ is proper 
(or goes to infinity in the terminology of \cite{GKM08}) precisely when the
linear part of the action of $\Gamma _0$ does not admit almost invariant vectors. Hence
our assumption on the finite index subgroups of $\Gamma$ translates to
the validity of the main assumption in that Theorem 
(and it takes care also of the ``no parallel orbits'' condition there). It is also
easily verified that the irreducibility of the
 unitary representation and the lack of fixed points for the affine action 
implies that this action is $C$-minimal (i.e., there is no closed, convex,
invariant, proper subset of the space).  
Hence Theorem 8.1 in \cite{GKM08} applies, and yields that 
the isometric $\Lambda$-action on $\mathcal{H}$ extends
continuously to $G$. By density of $\Lambda$, the $G$-action is also affine (and in particular its 
linear part extends $\pi$ continuously).
\end{proof}

We can now state and prove the property $(T)$ half of the main theorem:
\begin{proposition}\label{P:Thalf}
Let $\Gamma <_{c} \Lambda < G$ be as in Proposition \ref{P:amenhalf}, with $G$ compactly generated and $\Gamma$ (uniform or) square integrable. Let $N$ be a normal subgroup of $\Lambda$ such that
$\Gamma$ maps onto $\Lambda / N$ via the coset map,
and such that $\overline{[N,N]}$ (hence also $\overline N$) is co-compact in $G$.
Then $\Lambda / N$ has property $(T)$.
 \end{proposition}
 \begin{proof}
Recall that the reduced cohomology of $G$ with coefficients in $\pi$ (for a unitary representation $\pi : G \to \mathcal{U}(\mathcal{H})$ of $G$ on a Hilbert space) is $\overline{H^{1}}(G,\pi) = Z^{1}(G,\pi) / \overline{B^{1}(G,\pi)}$, the space of 1-cocycles modulo the closure of the subspace of 1-coboundaries (with respect to the topology of uniform convergence on compact sets).  The second 
author has shown \cite{Sha00} that property $(T)$ is equivalent to the vanishing of reduced cohomology for every {\it irreducible} unitary representation, provided the group is finitely (compactly) generated.
Note that while $\Lambda$ may not be finitely generated, since $\Gamma$ maps onto $\Lambda / N$ and $\Gamma$ is finitely generated, $\Lambda / N$ is finitely generated.  Let $\pi : \Lambda / N \to \mathcal{U}(\mathcal{H})$ be an irreducible unitary representation and suppose that $\overline{H^{1}}(\Lambda / N, \pi) \ne 0$.  We will obtain a contradiction to the existence of such a representation and therefore conclude that $\Lambda / N$ has property $(T)$.
Our method is to first show that such irreducible cohomological 
representation must be finite-dimensional, then show that in a 
cohomological finite-dimensional representation $\pi(\Gamma)$ must be finite, and finally show that finite representations support no 
cohomology.  These three steps appear as the three Lemmas below.
\end{proof}

%Recall that if $\Gamma$ is a finitely generated lattice in a locally compact second countable group $G$ then $\Gamma$ is {\bf square integrable} when there exists a fundamental domain $F$ for $G / \Gamma$ such that
%\[
%\int_{F} | \alpha(g,x) |^{2}~dm(x) < \infty
%\]
%where $\alpha : G \times F \to \Gamma$ is the cocycle given by $\alpha(g,x) = \gamma$ if and only if $gx\gamma \in F$, $| \cdot |$ denotes the word length in $\Gamma$ (the choice of generating set will not affect the finiteness of the integral) and $m$ is the finite Haar measure on $F$.  This requirement is imposed to utilize the results of \cite{Sha00} obtained via induced representations.

\begin{lemma}
Let $\Gamma <_{c} \Lambda < G$ and $N$ be as in Proposition \ref{P:Thalf}.
Let $\pi :  \Lambda / N \to \mathcal{U}(\mathcal{H})$ be an irreducible unitary representation such that $\overline{H^{1}}(\Lambda / N, \pi) \ne 0$.  Then $\pi$ is finite-dimensional.
\end{lemma}
\begin{proof}
Treat $\pi$ as a representation of $\Lambda$ with $\pi(N)$ being trivial.  Let $b \in Z^{1}(\Lambda,\pi)$ such that $b \ne [0]$ in $\overline{H^{1}}(\Lambda, \pi)$.  Then $b |_{\Gamma} \ne [0]$ in $\overline{H^{1}}(\Lambda,\pi)$ since $\Gamma$ maps onto $\Lambda / N$.
The second author has shown (Theorem 10.3 in \cite{Sha00}) that the latter
implies that there exists a nonzero $\Lambda$-invariant subspace on which the 
{\it linear} (but not necessarily the isometric -- this is why the Gelander-Karlsson-Margulis result will be needed below) 
$\Lambda$-action $\pi$ extends continuously to a unitary representation of $G$.  As $\pi$ is irreducible for $\Lambda$, the action extends on the entire space.  Continue to denote this new $G$-representation by $\pi$.
Since $\pi(N)$ is trivial, $\pi(\overline{N})$ is too.  
Therefore $\pi$ is in fact an irreducible representation of the compact 
group $G / \overline{N}$ (as $\overline{N}$ is co-compact), and is therefore finite-dimensional.
\end{proof}

\begin{lemma}\label{L:2}
Let $\Gamma <_{c} \Lambda < G$ and $N$ be as in Proposition \ref{P:Thalf}.
Let $\pi : \Lambda / N \to \mathcal{H}$ be a finite-dimensional, irreducible, unitary representation, such that $\overline{H^{1}}(\Lambda / N, \pi) \ne 0$.  Then $\pi(\Gamma)= \pi(\Lambda)$ is finite.
\end{lemma}
\begin{proof}
Treat $\pi$ as a $\Lambda$-representation, and let $b$ be the
assumed non-(cohomologically-)zero cocycle on $ \Lambda / N$. As usual,
$\pi$ and $b$ determine an affine, fixed point free isometric $\Lambda$-action
via  $\lambda v := \pi(\lambda)v + b(\lambda)$, having $N$ in its kernel.
  Since $\Gamma$ maps onto $\Lambda / N$, $\pi$ is an irreducible $\Gamma$-representation.
 We first show that {\it some} finite index subgroup $\Gamma_{0}<\Gamma$ 
has almost invariant vectors for $\pi$. Indeed, if this is not the case
then we are exactly in position to invoke  Proposition \ref{P:GKM}, which
yields that   the affine action of $\Lambda$ extends to a continuous affine 
isometric action of $G$.  Now $N$ is in the kernel of the isometric $\Lambda$-action, hence $\overline{N}$ is in the kernel of the $G$-action.  
But $\overline{N}$ is co-compact so the $G$-action factors through an 
action of a compact group and hence has a fixed point.  
Thus the $\Lambda$-action has a fixed point,
which contradicts our assumption on the cocycle $b$.

Let then  $\Gamma_{0}< \Gamma$ be a finite index subgroup having 
almost invariant vectors for $\pi$.  By passing to a finite index subgroup of it we may of course assume it is also normal in $\Gamma$. 
 As $\pi$ is finite-dimensional there is then a non-zero invariant vector 
for $\Gamma_{0}$ (compactness). The normality of $\Gamma_0$ implies  that 
for any $\Gamma$-action the set of $\Gamma_0$-fixed points is 
$\Gamma$-invariant. In our setting it follows from irreducibility 
that the $\Gamma$-invariant subspace of $\Gamma_0$-fixed vectors is (non-zero
and hence) all of $ \mathcal{H}$. Thus $\Gamma _0$ acts trivially and 
$\pi (\Gamma)= \pi (\Lambda)$ is finite, as claimed.
\end{proof}

\begin{lemma}
Let $\Gamma <_{c} \Lambda < G$ and $N$ be as in Proposition \ref{P:Thalf}.  Let $\pi : \Lambda / N \to \mathcal{H}$ be a finite unitary representation. 
Then $\overline{H^{1}}(\Lambda / N, \pi) = 0$.  

\end{lemma}
\begin{proof}
Assume to the contrary that there exists a non-(cohomologically-)zero cocycle
$b$ for the finite representation $\pi$ of the group $\Lambda / N$. 
Let $\Lambda _0 < \Lambda$ be the finite index kernel of $\pi$, and 
$\Gamma _0 = \Lambda _0 \cap \Gamma < \Gamma $. 
%Of course, $N<\Gamma _0 < \Lambda _0$. 
Since $\Lambda _0$ has finite index, $b |_{\Lambda _0}$ remains
non-zero. Being a 1-cocycle for the trivial $\Lambda _0$-representation, 
it defines a non-zero homomorphism of  $\Lambda _0$ to the additive vector 
group of the representation 
space, from which we get a nontrivial homomorphism 
$\varphi : \Lambda_{0} \to \mathbb{R}$.  The second author has shown (Theorem 0.8 in \cite{Sha00}) that, since $\Gamma_{0}$ is a finitely generated square integrable lattice in $G$, $\varphi |_{\Gamma_{0}}$ then extends to a 
homomorphism on $G_{0} = \overline{\Lambda_{0}}$ (note that the extension may not agree with $\varphi$ on all of $\Lambda_{0}$, just on $\Gamma_{0}$).
Call this extension $\psi : G_{0} \to \mathbb{R}$. It is of course 
still non-zero 
as $\psi |_{\Gamma_{0}} = \varphi \neq 0$ 
(noting that $\varphi (\Gamma _0)$ has finite index in $\varphi (\Lambda _0) $, thus
cannot be zero). 

Now, $\psi$ vanishes on $[N,N]$ hence also on $\overline{[N,N]}$.  
But $\overline{[N,N]}$ is co-compact in $G_{0}$, hence $\psi =0$, a 
contradiction.
This completes the proof of the Lemma, and with it the proof of Proposition \ref{P:Thalf}, and all of Theorem \ref{T:1}.
\end{proof}

\section{Proof of the Bijection of Commensurability Classes}

The second part of the Main Theorem follows immediately from Theorem \ref{T:1}
and the following general result.

\begin{proposition}\label{P:oneone}
Let $\Gamma < \Lambda$ be countable discrete groups such that $\Lambda$ commensurates $\Gamma$. Let $\varphi : \Lambda \to H$ be a dense homomorphism into a locally compact totally disconnected group $H$ such that $K = \overline{\varphi(\Gamma)}$ is compact open, and $\varphi^{-1}(K) = \Gamma$.  Then the map $N \mapsto \overline{\varphi(N)}$ induces a bijection between commensurability classes of normal subgroups $N \normal \Lambda$ with $[\Gamma : N \cap \Gamma] < \infty$, and commensurability classes of open normal subgroups of $H$.
\end{proposition}
\begin{proof}
Let $N$ be a normal subgroup of $\Lambda$ with $[\Gamma : N \cap \Gamma] < \infty$.  Then $[K : \overline{\varphi(\Gamma \cap N)}] < \infty$ and 
$\overline{\varphi(\Gamma \cap N)}$ is a compact open subgroup of $H$.  Since $\overline{\varphi(N)}$ contains this group, $\overline{\varphi(N)}$ is an open normal subgroup of $H$.

Let $N_{1}$ and $N_{2}$ be commensurate normal subgroups of $\Lambda$ with $[\Gamma : N_{1} \cap \Gamma], [\Gamma : N_{2} \cap \Gamma] < \infty$.  Then $N_{1} \cap N_{2}$ is a normal subgroup of $\Lambda$ that has finite index in both $N_{1}$ and $N_{2}$.  Therefore $\overline{\varphi(N_{1} \cap N_{2})}$ is an open normal subgroup of $H$ that has finite index in both $\overline{\varphi(N_{1})}$ and $\overline{\varphi(N_{2})}$ meaning that $N_{1}$ and $N_{2}$ are mapped to the same commensurability class of open normal subgroups.  Therefore the induced map 
on the commensurability classes is well defined.

Surjectivity is obvious: given an open normal subgroup $M$ of $H$,  set 
$N = \varphi^{-1}(M)$.  Then $N$ is normal in $\Lambda$ and $[\Gamma : N \cap \Gamma] < \infty$ since $M$ contains a finite index subgroup of $\overline{\varphi(\Gamma)}$.  Of course, $\overline{\varphi(N)} = \overline 
{\varphi (\Lambda) \cap M}= M$, as $M$ is open and $\varphi(\Lambda)$ is dense.

To prove injectivity, 
take $N_{1}$ and $N_{2}$ to be normal subgroups of $\Lambda$ with $[\Gamma : N_{1} \cap \Gamma], [\Gamma : N_{2} \cap \Gamma] < \infty$ such that $\overline{\varphi(N_{1})}$ and $\overline{\varphi(N_{2})}$ are commensurate (open normal) subgroups.  Since $\varphi$ is a homomorphism, $\varphi^{-1}(\overline{\varphi(N_{1})})$ and $\varphi^{-1}(\overline{\varphi(N_{2})})$ are commensurate 
subgroups of $\Lambda$.  Once we show that $[\varphi^{-1}(\overline{\varphi(N_{1})}) : N_{1}] < \infty$ and $[\varphi^{-1}(\overline{\varphi(N_{2})}) : N_{2}] < \infty$ we would get immediately that $N_{1}$ and $N_{2}$ are commensurate, 
implying injectivity. So, we are only left with proving that any $N$ as in the Proposition 
has finite index in $\varphi^{-1}(\overline{\varphi(N)})$.

Indeed, as $\varphi(N)$ is dense in $\overline{\varphi(N)}$ and $K=\overline{\varphi(\Gamma)}$ is open, $\overline{\varphi(N)} \subseteq K \varphi(N)$.  Set $Q = \overline{\varphi(N)} \cap \varphi(\Lambda)$.  For $h \in Q$, write $h = kn$ for some $k \in K$ and $n \in \varphi(N)$.  Then $hn^{-1} = k \in K$ and $hn^{-1} \in \varphi(\Lambda) \varphi(N) = \varphi(\Lambda)$ so $hn^{-1} \in 
K \cap \varphi(\Lambda) = \varphi(\Gamma)$ (since $\varphi^{-1}(K) = \Gamma$).  Therefore $Q \subseteq \varphi(\Gamma)\varphi(N) = \varphi(\Gamma N)$.  Since $\varphi$ is a homomorphism,
\[
[\varphi ^{-1} (Q) : \varphi ^{-1} (\varphi(N))] \leq [Q : \varphi(N)] \leq [\varphi(\Gamma N) : \varphi(N)] \leq [\Gamma N : N] = [\Gamma : \Gamma \cap N] < \infty.
\]

Because $\Gamma = \varphi ^{-1}(K)$, $\ker \varphi < \Gamma$ and we also have:
\[
[ \varphi ^{-1} (\varphi(N)) : N]=[N \ker \varphi :N] \leq [N \Gamma : N] < \infty
\]

These two finiteness results yield $ [\varphi^{-1}(Q) : N] < \infty$,
precisely what needed to be proved.
\end{proof}

%Now $\varphi(N) \simeq N / (N \cap \ker(\varphi))$ and $Q \simeq \varphi^{-1}(Q) / (\varphi^{-1}(Q) \cap \ker(\varphi)) = \varphi^{-1}(Q) / \ker(\varphi)$ so
%\[
%[\varphi^{-1}(Q) : N] \leq [Q : \varphi(N)][\ker(\varphi) : N \cap \ker(\varphi)].
%\]
%Since $\ker(\varphi) \subseteq \Gamma$ (as $\ker(\varphi) \subseteq \varphi^{-1}(K) = \Gamma$), $[\ker(\varphi) : \ker(\varphi) \cap N] \leq [\Gamma : \Gamma \cap N] < \infty$.  Therefore $[\varphi^{-1}(\overline{\varphi(N)}) : N] = [\varphi^{-1}(Q) : N] < \infty$, and the proof is complete.

%Finally, let $M$ be an open normal subgroup of $H$.  Set $N = \varphi^{-1}(M)$.  Then $N$ is normal in $\Lambda$ and $[\Gamma : N \cap \Gamma] < \infty$ since $M$ contains a finite index subgroup of $\overline{\varphi(\Gamma)}$.  Since $\overline{\varphi(N)} = M$, the correspondence is surjective.

\section{Irreducible Lattices in Products of Groups}\label{S:bs2}

\begin{proof}[Proof of Corollary \ref{T:bs2}]
Assume first that all three conditions hold. 
%(i) $\Lambda \cap \{ e \} \times H$ is finite, (ii) $H$ has no infinite index open normal subgroups and (iii) $(\mathrm{proj}_{G}~\Lambda) \cap N$ is finite for every closed normal non-cocompact $N \normal G$. 
 The first implies that $\mathrm{proj}_{G} : \Lambda \to G$ has a finite 
kernel (contained in $H$), which by density of $\Lambda$ is a normal subgroup 
of $H$. Replacing $\Lambda$ and $H$ with their quotient by 
$\ker(\mathrm{proj}_{G})$, we may and shall assume hereafter 
that this projection is faithful, so that we can naturally identify
elements of $\Lambda$ with their image in $G$ 
(it is immediate to verify that all the other assumptions remain intact, 
and that the conclusion for $ \Lambda / \ker(\mathrm{proj}_{G})$
implies it for $\Lambda $ itself).

Set $\Gamma = \Lambda \cap (G \times K)$ where $K$ is a compact open subgroup of $H$. As $K$ is open, $\proj_{K}~\Gamma$ is dense in $K$ (using that the 
projection of $\Lambda$ to $H$ is dense).
It is a general fact that when $L$ is co-compact (a lattice) 
in a locally compact group $M$, and $U$ is an open subgroup of $M$, $L \cap U$ is co-compact (a lattice) in $U$.  
Applying this to $\Lambda$ we find that $\Gamma$ is co-compact in $G \times K$.
Since $K$ is compact, the projection of $\Gamma$ to $G$ is discrete and co-compact in $G$.

As $K$ is commensurated by $H$, being a compact open subgroup of it, 
$\Gamma$ is commensurated by $\Lambda$.  The projection of $\Lambda$ to $G$ is dense and therefore $\proj_{G}~\Gamma < \proj_{G}~\Lambda < G$ satisfy the general setup of our Main Theorem, where condition (iii) ensures that its main
assumption is satisfied. It follows that every  
$N \normal \Lambda$ contains a finite index subgroup of $\Gamma$. 

%Note that if $\Lambda $ is cocompact then so is 
%$\Gamma$ (in $G$), and everything works smoothly. 
%However, otherwise the lattice $\Gamma$ may not be square integrable, 
%and in fact may even fail to be finitely generated. We remark on the modofication of the proof needed in that case at the end.

Finally, observe that $\varphi = \proj_{H}$ satisfies the assumption in the
second part of the conclusion of the main Theorem. Condition (ii) in the Corollary says that there is only one commensurability class of open normal subgroups
of $H$ -- that of $H$ itself, thus the Main Theorem implies that 
there is only one for $\Lambda$ as well, which must be the class of those
normal subgroups having finite index. This proves the main direction
of the Corollary.
 
%however the necessity of these requirements is to apply the results of \cite{Sha00} which require them only in order to induce $\Gamma$-representations to $G$-representations.  In our situation, we may instead induce the $\Lambda$-representation to $G \times H$ and then consider the restriction of the representation to $G$ and obtain a $G$-representation with $\Gamma$ and $\Lambda$ (more precisely, the faithful projections of them) in the correct places.

%Let $\tau_{H,K} : K \to \Symm(H / K)$ define the relative profinite completion $\rpf{H}{K}$ (see section \ref{S:relprofcomp}).
%Observe that $\rpf{\Lambda}{\Gamma} \simeq H / \ker(\tau_{H,K})$ since $\proj : \Lambda \to H$ is a homomorphism such that $\overline{\proj(\Lambda)} = H$ and $\proj^{-1}(K) = \Lambda \cap (G \times K) = \Gamma$ hence Proposition \ref{P:relprofunivprop} implies $\rpf{\Lambda}{\Gamma}$ is isomorphic to $\rpf{H}{K}$ which in turn is isomorphic to $H / \ker(\tau_{H,K})$.  But $\ker(\tau_{H,K}) \normal H$ is compact (since $\ker_{H,K}$ is contained in $K$ which is compact) hence $\rpf{\Lambda}{\Gamma}$ is isomorphic to $H$ modulo a compact normal subgroup.  Since $H$ has no infinite index open normal subgroups neither does this quotient.
%Therefore, by Theorem \ref{T:oneone}, $\Lambda \simeq \proj_{G}~\Lambda$ has no infinite index infinite normal subgroups.

Now assume in the reverse direction that every infinite normal subgroup 
of $\Lambda$ has finite index. 
Observe that $N = \Lambda \cap (\{ e \} \times H)$ is normal in $\Lambda$ so 
it is either finite, or has finite index.  If the latter holds 
then the projection of $\Lambda$ to $G$ is finite, hence so is $G$,
contradicting the assumption that it is not a compact extension of an abelian group.
Thus condition (i) must hold. 

Next, suppose towards a contradiction that $H$ has an infinite index open normal subgroup $M$ (in particular, $H$ is infinite). Then  $\mathrm{proj}_{H}\Lambda  \cap M$ is
an infinite index normal subgroup of $\mathrm{proj}_{H}\Lambda $ (which is dense in $M$), 
hence 
its inverse image has the same property back in $\Lambda$. 
By our assumption on $\Lambda$ that inverse image
is finite, thus $M$ itself must be finite, and $H$ is discrete. It follows that
 $\Lambda$ must 
project onto $H$, and as $H$ is infinite the kernel of that projection cannot
have finite index, so by our assumption on $\Lambda$ it is finite. But this kernel is
a lattice in $G$ (in order for $\Lambda $ to be a lattice in $G \times H$), 
hence $G$ is compact, a contradiction. This proves the necessity of (ii).

Finally, let $M \normal G$ be a closed non-cocompact normal subgroup of $G$.  
A previous argument showed that our assumption implies that 
the map $\mathrm{proj}_{G} : \Lambda \to G$ has finite kernel, 
so $\Lambda_{0} = \mathrm{proj}_{G}~\Lambda$ also has the property 
that every infinite normal subgroup has finite index.
Set $N = \Lambda _0 \cap M$.  Then $N \normal \Lambda_{0}$, so $N$ is finite or has finite index in $\Lambda_{0}$.  If $N$ had finite index 
then so did
 $\overline{N} \normal \overline{\Lambda_{0}} = G$,  hence $M$ was co-compact.  Therefore the projection of $\Lambda$ to $G$ intersects every closed non-cocompact normal subgroup of $G$ finitely, which proves the necessity of 
condition (iii) and completes the proof of the whole Corollary.
\end{proof}

The case where $\Lambda$ is a non-uniform lattice is more involved due to the 
complication in the property (T) half of the proof, and 
requires a modification
of both the statement and the argument. 
Even if $\Lambda$ were assumed square integrable, the lattice $\Gamma <G$ may not 
have this property (e.g., when $\Lambda = SL_2(\mathbb{Z}[{\frac {1}{p}}])<  
(G=SL_2(\mathbb{R})) \times (H= SL_2(\mathbb{Q} _p))$, the lattice
$\Gamma = SL_2(\mathbb{Z})< G$ is not square integrable, even though $\Lambda$ 
is \cite{Sha00}). One can even construct such examples where $\Gamma$ is not finitely generated (taking $\Lambda$ as an irreducible non-uniform lattice in a product of rank one simple algebraic groups over a local field of positive 
characteristics). The simplest and most direct way to deal with this issue is to 
assume that $\Gamma$ itself be square integrable. In that case the whole argument (and result) goes through as is. However, with  a bit more structural assumptions one can do by assuming the square integrability of $\Lambda$. 
Here one has to modify the whole
property (T) argument so that instead of inducing unitary representations 
from $\Gamma$ to $G$, they are induced from $\Lambda$ to $G \times H$, and
then restricted to $G$. This requires getting into the proofs of some results 
from \cite{Sha00} and since we anyway do not currently have any concrete applications in mind, we prefer not to elaborate further on
this matter.
 
Finally, it is perhaps worth illustrating here how the lack of compact
generation assumption on $H$ yields useful additional flexibility. 
Let $K$ be a global field, and $\mathbf{G}$ be a simply connected, simple algebraic group defined over $K$. Then $\mathbf{G}(K)$ is a lattice
in $\mathbf{}(\mathbb{A})$, where $\mathbb{A}$ is the ring of adeles over $K$.
One may then decompose $\mathbf{G}(\mathbb{A})= G \times H$ where $G$ is a product over finitely many places including all the Archimedian ones, so that
its $S$-rank is at least 2, and
$H$ is the (restricted) direct product over all other places. The fact that
$\Lambda$ has dense projections follows from strong approximation,
while the fact that $H$ has no open normal subgroups follows from the 
Kneser-Tits conjecture over local fields. Here $\Gamma < G$ is square
integrable by \cite{Sha00}, so one deduces that every infinite normal
subgroup of $\mathbf{G}(K)$ has finite index. 
When  $\mathbf{G}$ is $K$-isotropic this can
then be easily upgraded to simplicity (modulo the center), but in the
anisotropic case the latter no longer holds in general. Of course, such normal
subgroup theorems can be similarly proved for $S$-arithmetic groups
when $S$ is any infinite set of places (including the Archimedian ones)
-- see also Section 2 in Chapter VIII of \cite{Ma91}. Note that 
the assumption that
$\mathbf{G}$ be simply connected is crucial when $S$ infinite. One can still
invoke our strategy when it isn't, but here
$H$ may contain open normal subgroups which 
classify, using the Main Theorem, the abstract normal subgroups of the corresponding
$S$-arithmetic group.

\normalsize
%\bibliography{references}
\providecommand{\bysame}{\leavevmode\hbox to3em{\hrulefill}\thinspace}
\providecommand{\MR}{\relax\ifhmode\unskip\space\fi MR }
\providecommand{\MRhref}[2]{%
  \href{http://www.ams.org/mathscinet-getitem?mr=#1}{#2}
}
\providecommand{\href}[2]{#2}

\end{document}